\documentclass{amsart}[12pt]
\usepackage{amssymb,amsmath}
\usepackage{enumerate}
\usepackage[english]{babel}
\usepackage{color}

\newtheorem{teo}{Theorem}[section]
\newtheorem{pro}{Proposition}[section]
\newtheorem{cor}{Corollary}[section]

\theoremstyle{definition}

\title[The classification of $\ast$-algebras of  $LOG$-integrable functions]{Isomorphic classification of $\ast$-algebras of  log-integrable functions}
\begin{document}
\date{June 17, 2017}
\begin{abstract}
Using the notion of passport of a normed Boolean algebra, necessary and sufficient conditions for 
a $\ast$-isomorphism of $\ast$-algebras of log-integrable measurable functions are found.
\end{abstract}

\author{Rustam Abdullayev }
\address{Uzbek State University of World Languages\\ Tashkent,  700174, Uzbekistan}
\email{arustambay@yandex.com}

\author{Vladimir Chilin}
\address{National University of Uzbekistan\\ Tashkent,  700174, Uzbekistan}
\email{vladimirchil@gmail.com}

\maketitle

\section{Introduction}

Let \ $(\Omega, \mathcal A, \mu)$ be a $\sigma$-finite measure a measure space, and let
$\mathcal L_{\log}(\Omega, \mathcal A, \mu)$ be the symmetric function space
consisting of complex valued measurable functions $f$ on  $(\Omega, \mathcal A, \mu)$ 
such that $\int_{\Omega}\log(1+|f|)\,d\nu<\infty$.
It is know that $\mathcal L_{\log}(\Omega, \mathcal A, \mu)$ is a $\ast$-subalgebra in the $\ast$-algebra $\mathcal L_0(\Omega, \mathcal A, \mu)$ of complex valued measurable functions on  $(\Omega, \mathcal A, \mu)$ 
(functions equal $\mu$-almost everywhere are identified)  \cite{dsz}. In addition, the space 
$\mathcal L_{\log}(\Omega, \mathcal A, \mu)$  is a non-locally-convex Hausdorff topological 
$\ast$-algebra with respect to the $F$-norm  $$\|f\|_{\log}=\int_{\Omega} log(1+|f|) d \mu.$$

In the case where the measure space $(\Omega, \mathcal A, \mu)$ is the unit circle in the complex plane endowed with Lebesgue measure, the boundary values of Nevanlinna functions belong to the symmetric function space $\mathcal L_{\log}(\Omega, \mathcal A, \mu)$,  and the map assigning a Nevanlinna
function its boundary values yields
an injective and continuous algebraic homomorphism from the Nevanlinna class to $\mathcal L_{\log}(\Omega, \mathcal A, \mu)$. Since the Nevanlinna class is not well behaved under the usual metric, 
it is natural to study its topological properties with respect to the $F$-norm  $\|\cdot\|_{\log}$ (see \cite{dsz}).

Taking into account various applications of $\ast$-algebras  $\mathcal L_{\log}(\Omega, \mathcal A, \mu)$ in the theory of functions of a complex variable, it is important to describe these algebras associated with different measures up to $\ast$-isomorphisms. This paper is devoted to solving this problem.
Utilizing the notion of passport of a normed Boolean algebra, we give necessary and sufficient conditions for a $\ast$-isomorphism of $\ast$-algebras of log-integrable measurable functions. The proof uses the method of papers \cite {a}, \cite {ach}, which give a description of the $\ast$-isomorphisms of Arens algebras of measurable functions.

\section{Preliminaries}

Let $(\Omega, \mathcal A, \mu)$ be $\sigma$-finite measure space, and let $\mathcal  L_0(\Omega, \mathcal A, \mu)$ 
be the $\ast$-algebra of complex valued measurable functions on $(\Omega, \mathcal A, \mu)$ (functions equal $\mu$-almost everywhere are identified). Following \cite{dsz}, we consider the $\ast$-subalgebra
 $$
 \mathcal L_{\log}(\Omega, \mathcal A, \mu)=\{f \in  \mathcal L_0(\Omega, \mathcal A, \mu): \int \limits_{\Omega} \log(1+|f|)d\mu< + \infty\} $$
of the $\ast$-algebra $\mathcal  L_0(\Omega, \mathcal A, \mu)$.
Since for any $a, b >0, \ a \neq 1, \ b \neq 1$, there are constants $0< d_1 < d_2$ such that $d_1 \log_a c \leq \log_b c \leq d_2 \log_a c$ for all $c>0$, it follows that  the definition of $\mathcal L_{\log}(\Omega, \mathcal A, \mu)$  does not depend on the choice of base of the logarithm.

For each $f \in \mathcal L_{\log}(\Omega, \mathcal A, \mu)$ we put
$$
\|f\|_{\log} = \int_{\Omega} \log (1+ | f |) d \mu.
$$
According to \cite [Lemma 2.1] {dsz}, the function 
$$
\|\cdot\|_{\log}: \mathcal L_{\log}(\Omega, \mathcal A, \mu) \rightarrow [0,\infty)
$$ 
is an $F$-norm, that is, given $f,g\in \mathcal L_{\log}(\Omega, \mathcal A, \mu)$,
\begin{enumerate}[(i)]
\item $\|f\|_{\log}>0$ if $f \neq 0$;

\item $\|\alpha f\|_{\log}\le\|f\|_{\log}$ if $|\alpha|\le 1$;

\item $\lim_{\alpha\to 0}\|\alpha f\|_{\log}=0$;

\item $\|f+g\|_{\log}\le\|f\|_{\log}+\|g\|_{\log}$.
\end{enumerate}
Besides, $\mathcal L_{\log}(\Omega, \mathcal A, \mu))$ is a complete topological $\ast$-algebra with respect to 
the topology generated by the metric $\rho(f,g)=\|f-g\|_{\log}$; see \cite[Corollary 2.7]{dsz}.

Let $\mu$ and $\nu$ be $\sigma$-finite measures on a measure space $(\Omega,
\mathcal A)$ such that $\mu \sim \nu$, that is,
$$
\mu(A) =0 \ \Leftrightarrow\ \nu(A) =0, \ A \in  \mathcal A.
$$
In this case,
$$
\mathcal L_0(\Omega, \mathcal A, \mu) = \mathcal L_0(\Omega, \mathcal A, \nu):=\mathcal L_0(\Omega), \ \ \mathcal L_{\infty}(\Omega, \mathcal A, \mu) = \mathcal L_{\infty}(\Omega, \mathcal A,  \nu) :=\mathcal L_{\infty}(\Omega).
$$
We denote by $\frac{d\nu}{d\mu}$ the
Radon-Nikodym derivative of the measure $ \nu $ with respect to the measure $\mu$. It is known that $\frac{d\nu}{d\mu} \in \mathcal  L_0(\Omega, \mathcal A, \mu)$ and
$$
f \in \mathcal  L_1(\Omega, \mathcal A, \nu) \ \Leftrightarrow \ f\cdot \frac{d\nu}{d\mu} \in \mathcal  L_1(\Omega, \mathcal A, \mu);
$$
in addition,
$$\int_{\Omega}\ f d\nu =\int_{\Omega}(f \cdot \frac{d\nu}{d\mu}) d\mu.
$$
Since $\mu \sim \nu$, it follows that there exists an inverse function $(\frac{d\nu}{d\mu})^{-1} = \frac{d\mu}{d\nu}$.

\begin{pro}\label{p1}
The following conditions are equivalent:
\begin{enumerate}[(i)]
\item $\frac{d\nu}{d\mu} \in \mathcal L_{\infty}(\Omega)$;
\item $\mathcal L_{\log}(\Omega, \mathcal A, \mu) \subset \mathcal L_{\log}(\Omega, \mathcal A, \nu)$.
\end{enumerate}
\end{pro}
\begin{proof}  $(i)\ \Rightarrow\ (ii)$: If $\frac{d\nu}{d\mu} \in \mathcal L_{\infty}(\Omega)$ and 
$f \in \mathcal L_{\log}(\Omega, \mathcal A, \mu)$, then
$$\int_{\Omega}\log(1+|f|)d\nu=\int_{\Omega}\log(1+|f|) \frac{d\nu}{d\mu}d\mu \atop \leq\|\frac{d\nu}{d\mu}\|_{\infty}\int_{\Omega} \log(1+|f|)d\mu< \infty.$$
Consequently, $f \in \mathcal L_{\log}(\Omega, \mathcal A, \nu)$, hence  $\mathcal L_{\log}(\Omega, \mathcal A, \mu) \subset \mathcal L_{\log}(\Omega, \mathcal A, \nu)$.

$(ii)\ \Rightarrow\ (i) $: Let $\mathcal L_{\log}(\Omega, \mathcal A, \mu) \subset \mathcal L_{\log}(\Omega, \mathcal A, \nu)$ and suppose that $\frac{d\nu}{d\mu} \notin  \mathcal L_{\infty}(\Omega)$.
In this case there exists a sequence of positive integers $ n_k \uparrow \infty $ such that
$\mu(A_{n_k}) >0$, where $A_{n_k}=\{\omega \in \Omega: n_k \leq \frac{d\nu}{d\mu}(\omega)< n_k+1\}$,
for all natural $k$.
Consider a step measure function  $g=\sum_{k=1}^{\infty}\frac{1}{k^{2}\cdot\mu(A_{n_{k}})}\cdot \chi_{A_{n_{k}}}$, where $\chi_A$  is the characteristic function of a set  $A \in \mathcal A$, that is, $\chi_A(\omega) = 1$ for  $\omega \in A$ and $\chi_A(\omega)=0$, if \ $\omega \notin A$.
Putting $f=e^{g}-1$, we obtain that
\begin{equation}\label{e10}
\int_{\Omega}\ln(1+|f|)d\mu=\sum_{k=1}^{\infty}\frac{\mu(E_{n_{k}})}{k^{2}\mu(E_{n_{k}})}=
\sum_{k=1}^{\infty}\frac{1}{k^{2}}<\infty,
\end{equation}
At the same time,
\begin{equation}\label{e20}
\int_{\Omega}\log(1+|f|)d\nu=\int_{\Omega}\log(1+|f|)\frac{d\nu}{d\mu}d\mu \geq
\sum_{k=1}^{\infty}\frac{n_{k}\cdot
\mu(E_{n_{k}})}{k^{2}\mu(E_{n_{k}})}=$$
$$\sum_{k=1}^{\infty}\frac{n_{k}}{k^{2}}\geq\sum_{k=1}^{\infty}\frac{1}{k}=\infty .
\end{equation}
It follows from (\ref{e10}) and (\ref{e20})  that $f\in \mathcal L_{\log}(\Omega, \mathcal A, \mu)$ and $f\notin \mathcal L_{\log}(\Omega, \mathcal A, \mu)$, that is, $\mathcal L_{\log}(\Omega, \mu)$ is not a subset of $\mathcal L_{\log}(\Omega, \nu)$ which contradicts the assumption. Consequently, $\frac{d\nu}{d\mu} \in \mathcal L_{\infty}(\Omega)$.
\end{proof}
Directly from Proposition \ref{p1} we obtain the following.
\begin{cor}\label{c1}
The following conditions are equivalent:
\begin{enumerate}[(i)]
\item $\frac{d\nu}{d\mu} \in \mathcal L_{\infty}(\Omega)  \  \ \text{and} \ \  \frac{d\mu}{d\nu} \in \mathcal L_{\infty}(\Omega)$;
\item $\mathcal L_{\log}(\Omega, \mathcal A, \mu) = \mathcal L_{\log}(\Omega, \mathcal A, \nu)$.
\end{enumerate}
\end{cor}

Let $\nabla_\mu$  be the complete Boolean algebra of equivalence classes $e=[A]$ of $\mu$-almost everywhere
equal sets. It is known that $\widehat{\mu}(e) = \mu(A)$  is a strictly positive countably additive 
$\sigma$-finite measure on $\nabla_ \mu$.
Since $\mu \sim \nu $, it follows that  $\nabla_ \mu=\nabla_ \nu :=\nabla$. In what follows, the measure $\widehat {\mu}$ 
will be denoted by $\mu$, and $\mathcal  L_0(\Omega)$ ($\mathcal L_\infty(\Omega)$) by
$\mathcal  L_0(\nabla)$  (respectively, $\mathcal L_\infty(\nabla)$).

Let $\varphi: \nabla \rightarrow \nabla$ be an arbitrary automorphism of the Boolean algebra $ \nabla $. It is clear that $\lambda(e) =\mu(\varphi(e))$, $e \in \nabla$, is a strictly positive countably additive $\sigma $-finite measure on the Boolean algebra $\nabla$. Denote by $\Phi$ a $\ast$-isomorphism of the $\ast$-algebra $\mathcal L_0(\nabla)$ such that $\varphi (e)=\Phi(e)$ for all $e \in \nabla$. The restriction of $\Phi$ on the $C^*$-algebra $\mathcal L_\infty(\nabla)$ is a $\ast$-isomorphism of  $\mathcal L_\infty(\nabla)$.

Since
$$
\int_{\Omega}\sum_{i=1}^{n} c_i e_i \ d\lambda = \sum_{k=i}^{n} c_i\lambda (e_i)= \sum_{k=i}^{n} c_i \mu(\varphi(e_i))=\int_{\Omega}\Phi(\sum_{i=1}^{n}c_ie_i) \ d\mu
$$
for all $e_i \in \nabla, \ \mu(e_i) < \infty, \ e_ie_j =0, \ i\neq j, \ i,j=1,..., n$, it follows that
$$
\int_{\Omega}f d\lambda = \int_{\Omega}\Phi(f) d\mu, \ \ \int_{\Omega}\Phi^{-1}(g) d\lambda =\int_{\Omega}g d\mu
$$
for all  $f \in \mathcal L_{1}(\nabla, \lambda) and g \in \mathcal L_{1}(\nabla, \nu)$. This means that $\Phi(\mathcal L_{1}(\nabla, \lambda))=\mathcal L_{1}(\nabla, \mu)$.

Let $\mu$ ($\nu$)  be a strong  positive countably additive  $\sigma$-finite measure  on a complete Boolean algebra  
$\nabla_1$ (respectively, $\nabla_2$), let $\varphi: \nabla_1 \rightarrow \nabla_2$ be an isomorphism, and let 
$\Phi: \mathcal  L_0(\nabla_1) \to \mathcal  L_0(\nabla_2)$  be a $\ast$-isomorphism  such that 
$\varphi(e) = \Phi(e)$ for all $e \in \nabla_1$. It is clear that $\lambda(\varphi(e)) =\mu(e), \ e \in \nabla$, 
is a strong  positive countably additive  $\sigma$-finite measure on the Boolean algebra  $\nabla_2$.

\begin{pro}\label{p2}
$\Phi(log(1+|f|)) = log(1+\Phi(|f|))$ for all $f \in \mathcal L_0(\nabla_1)$ and
$\Phi((\mathcal L_{\log}(\nabla_1, \mu)) = \mathcal L_{\log}(\nabla_2, \lambda)$.
\end{pro}
\begin{proof}
The restriction $\Psi$ of the $\ast$-isomorphism $\Phi$ on the $C^*$-algebra $\mathcal L_\infty(\nabla_1)$ is a $\ast$-isomorphism from $C^*$-algebra $\mathcal L_\infty(\nabla_1)$ onto $C^*$-algebra $\mathcal L_\infty(\nabla_2)$. Then 
we have 
$$
\Psi(u\circ|g|) = u\circ\Psi(|g|)
$$
for any $g \in \mathcal L_\infty(\nabla_1)$ and every continuous function $u:[0,+\infty) \rightarrow \mathbb R$.
The function $u(t)=log(1+t)$ is continuous on the interval $[0, + \infty)$. Therefore
$$
\Phi(log(1+|g|)) = \Psi(log(1+|g|))= log(1+\Psi(|g|) =log(1+\Phi(|g|))
$$
for all  $g \in \mathcal L_\infty(\nabla_1)$. If $f \in \mathcal L_0(\nabla_1)$,  then setting 
$g_n = |f|\cdot\chi_{\{|f|\leq n\}}$, we obtain $g_n \in \mathcal L_\infty(\nabla_1), \ n \in \mathbb N$, 
$0 \leq g_n \uparrow |f|$ and $log(1+g_n) \uparrow log(1+|f|)$. Since  the isomprphism $\Phi: \mathcal L_0(\nabla_1) \rightarrow \mathcal L_0(\nabla_2)$ is order preserving, we have
$$
\Phi(g_n) \uparrow \Phi(|f|), \ \ log(1+\Phi(g_n)) \uparrow log(1+\Phi(|f|)),
$$
and
$$log(1+\Phi(g_n)) = \Phi(log(1+g_n))\uparrow \Phi(log(1+|f|)).$$
hence
$$
\Phi(log(1+|f|)) = log(1+\Phi(|f|))
$$
for all  $f \in \mathcal L_0(\nabla_1)$.

By the definition of the  $\ast$-algebra $\mathcal L_{\log}(\nabla_1, \mu)$,
$$
f \in \mathcal L_{\log}(\nabla_1, \mu) \Leftrightarrow  (f \in \mathcal L_0(\nabla_1) \ \ \text{and} \ \  \log(1+|f|) \in \mathcal L_{1}(\nabla_1, \mu)).
$$
Consequently, in view of
$$\Phi(\mathcal L_{1}(\nabla_1, \mu))=\mathcal L_{1}(\nabla_2, \lambda)
 $$
and
$$\Phi(log(1+|f|)) = log(1+\Phi(|f|)), \ \ f \in \mathcal L_0(\nabla_1),
$$
we conclude that
$$
log(1+|\Phi(f)| = log(1+\Phi(|f|) \in \mathcal L_{\log}(\nabla_2, \lambda).
$$
Therefore $\Phi(f) \in \mathcal L_{\log}(\nabla_2, \lambda)$ for all $f \in \mathcal L_{\log}(\nabla_1, \mu)$. Similarly, using the inverse $\ast$-isomorphism $\Phi^{-1}$, we see that $\Phi^{-1}(h) \in \mathcal L_{\log}(\nabla_1, \mu)$ for all $h \in \mathcal L_{\log}(\nabla_2, \lambda)$. Therefore $\Phi(\mathcal L_{\log}(\nabla_1, \mu)) = \mathcal L_{\log}(\nabla_2, \lambda)$.
\end{proof}

Let $\mu$ ($\nu$)  be a strong  positive countably additive $\sigma$-finite measure on a complete Boolean algebra  
$\nabla_1$ (respectively, $\nabla_2$). The measures $\mu$ and $\nu$  are called log-equivalent if there exists an isomorphism  $\varphi: \nabla_1 \rightarrow \nabla_2$ such that $\mathcal L_{\log}(\nabla_2, \nu)=\mathcal L_{\log}(\nabla_2, \mu \circ \varphi),$ where $(\mu \circ \varphi)(e) = \mu(\varphi(e)), \ e \in \nabla_2$.

As noted above,  $\lambda =\mu \circ \varphi$ is a strictly positive and countably additive $\sigma$-finite measure on the Boolean algebra $\nabla_2$. Therefore, by virtue of Corollary \ref{c1}, condition
$\mathcal L_{\log}(\nabla, \nu)=\mathcal L_{\log}(\nabla, \mu \circ \varphi)$ is equivalent to the system
$\frac{d\nu}{d\lambda} \in \mathcal L_{\infty}(\nabla_2)$ and $\frac{d\lambda}{d\nu} \in \mathcal L_{\infty}(\nabla_2)$.

\begin{teo}\label{t1}
Let $\mu$ ($\nu$)  be a strong  positive countably additive  $\sigma$-finite measure  on a complete Boolean algebra  
$\nabla_1$ (respectively, $\nabla_2$). Then the algebras  $\mathcal L_{\log}(\nabla_1, \mu)$ and $\mathcal L_{\log}(\nabla_2, \nu)$ are $\ast$-isomorphic if and only if the measures $\mu$ and $\nu$ are log-equivalent.
\end{teo}

\begin{proof}
Let  measures $\mu$ and $\nu$ be log-equivalent, that is, there exists an isomorphism $\varphi: \nabla_1 \rightarrow \nabla_2$  such that $\mathcal L_{\log}(\nabla_2, \nu)=\mathcal L_{\log}(\nabla_2, \mu \circ \varphi)$. Let $\Phi:\mathcal  L_0(\nabla_1) \to \mathcal  L_0(\nabla_2)$ be  a $\ast$-isomorphism for which $\varphi(e) = \Phi(e)$ for all $e \in \nabla$. By Proposition \ref{p2}, we have
$$\Phi(\mathcal L_{\log}(\nabla_1, \mu))=(\mathcal L_{\log}(\nabla_2, \mu \circ \varphi)) 
= \mathcal L_{\log}(\nabla_2, \nu).
$$
This means that the algebras  $\mathcal L_{\log}(\nabla_1, \mu)$ and $\mathcal L_{\log}(\nabla_2, \nu)$ are $\ast$-isomorphic.

Conversely, suppose that the algebras $\mathcal L_{\log}(\nabla_1, \mu)$ and $\mathcal L_{\log}(\nabla_2, \nu)$ are 
$\ast$-isomorphic, that is, there exists a $\ast$-isomorphism 
$$
\Psi: \mathcal L_{\log}(\nabla_1, \mu) \to \mathcal L_{\log}(\nabla_2, \nu).
$$
Let $\{e_n\}_{n=1}^\infty$ be a partition of a unity of the Boolean algebra $\nabla_1$ such that $$\mu(e_n) < \infty 
\text{ \ and \ } \nu(\Psi(e_n)) < \infty \ \ \text{for all} \ \ n.
$$
It is clear that  $\Psi: \mathcal L_{\log}(e_n \nabla_1, \mu) \to  \mathcal L_{\log}(\Psi(e_n) \nabla_2, \nu)$ is a $\ast$-isomorphism. Since $e_n \nabla_1 \subset  \mathcal L_{\log}(e_n \nabla_1, \mu)$ and  
$\Psi(e_n) \nabla_2 \subset  \mathcal L_{\log}(\Psi(e_n) \nabla_2, \nu)$, it follows that the restriction 
$\varphi_n$ of that $\ast$-isomorphism $\Psi:  \mathcal L_{\log}(e_n \nabla_1, \mu) \to  \mathcal L_{\log}(\Psi(e_n) \nabla_2, \nu)$ on the Boolean algebra $e_n \nabla_1$ is an isomorphism from  $e_n \nabla_1$ onto $\Psi(e_n) \nabla_2$. Define the map $\varphi:  \nabla_1 \to  \nabla_2$ by the formula
$$
\varphi(e) = \sup\limits_{n\geq 1} \varphi_n(e\cdot e_n), \ e \in \nabla_1.
$$
It is clear that $\varphi$ is an isomorphism from  $\nabla_1$ onto $\nabla_2$ and $\Psi(e) = \varphi(e)$ for all $e \in \nabla_1$. This means that (see Proposition \ref{p2})
$$\mathcal L_{\log}(\nabla_2, \mu \circ \varphi) = \Psi(\mathcal L_{\log}(\nabla_1, \mu)) = \mathcal L_{\log}(\nabla_2, \nu),$$
that is, the measures $\mu$ and $\nu$ are log-equivalent.
\end{proof}

Let $ \nabla$ be an arbitrary complete Boolean algebra, $\nabla_e=\{g \in \nabla:g \leq e\}$, where $ 0\neq e \in \nabla$. 
Denote by $\tau(\nabla_e)$ the minimum cardinality of a set that is dense in $ \nabla_e$ with respect to the order topology ($(o)$-topology). An infinite Boolean algebra $ \nabla$ is said to be homogeneous if $\tau(\nabla_e)=\tau(\nabla_g)$ for any nonzero $ e, g \in \nabla$. The cardinality of $\tau(\nabla)$ is called the weight of the homogeneous Boolean algebra $\nabla$ (see, for example, \cite[chapter VII]{v}).

Let $\mathbf 1_{\nabla} $ be the unity in a Boolean algebra $\nabla$.
It is known  that any infinite Boolean algebra $(\nabla, \mu)$ with $\mu(\mathbf 1_{\nabla})<\infty$ is a direct product of homogeneous Boolean algebras $ \nabla_{e_n}$, $e_n\cdot e_m = 0, \ n\neq m$, 
$\tau_n=\tau (\nabla_{e_n})< \tau_{n+1}$ (\cite[Chapter VII, \S 2, Theorem 3]{v}). Set $\mu_n = \mu(e_n)$. The matrix
$\left(
                  \begin{array}{ccc}
                    \tau_{1} & \tau_{2} & \dots \\
                    \mu_1 & \mu_2 & \dots \\
                  \end{array}
                \right)$
is called  the  passport of the Boolean algebra $(\nabla, \mu)$.

The following theorem gives a classification of Boolean algebras with finite measure \cite[Chapter VII, \S 2, Theorem 5]{v}).
\begin{teo}\label{t3} Let $\mu$ ($\nu$)  be a probability measure  on  infinite complete Boolean algebra  $\nabla_1$  (respectively, $\nabla_2$). Let $\left(
\begin{array}{ccc}
                   \tau_{1}^{(1)} & \tau_{2}^{(1)} & \dots  \\
                    \mu_1 & \mu_2 & \dots \\
                  \end{array}
                \right)$ be the passport of the Boolean algebra  $(\nabla_1, \mu)$, and let $\left(
                  \begin{array}{ccc}
                   \tau_{1}^{(2)} & \tau_{2}^{(2)} & \dots  \\
                    \nu_1 & \nu_2 & \dots \\
                  \end{array}
                \right)$ be the passport of the Boolean algebra  $(\nabla_2, \nu)$.
The following conditions are equivalent:
\begin{enumerate}[(i)]
\item There exists an isomorphism $\varphi:  \nabla_1 \to  \nabla_2$ such that $\mu(e) = \nu(\varphi (e))$ for all   
$e \in \nabla_1$;
\item $\tau_{n}^{(1)} =\tau_{n}^{(2)}$ and  $\mu_n =\nu_n$  for all $n$.
\end{enumerate}
\noindent
In addition, the Boolean algebras  $\nabla_1$  and $\nabla_2$ are isomorphic if and only if the upper rows of their passports 
coincide.
\end{teo}

Now we are ready to give a criterion for a $\ast$-isomorphism between $\ast$-algebras $\mathcal L_{\log}(\nabla, \mu)$ of log-integrable measurable functions associated with finite measure spaces. Let $\mu$ and $\nu$  be finite measures on complete Boolean algebras  $\nabla_1$  and $\nabla_2$, respectively. Since  $\mathcal L_{\log}(\nabla, \mu) =  \mathcal L_{\log}(\nabla, \frac{\mu}{\mu(\mathbf 1_{\nabla})})$, we can assume without loss of generality that
$\mu(\mathbf 1_{\nabla_1}) = 1 =\nu(\mathbf 1_{\nabla_2})$.

\begin{teo}\label{t4} Let $\mu$ ($\nu$)  be a probability measure  on  an infinite complete Boolean algebra  $\nabla_1$  (respectively, $\nabla_2$), and let $\left(
\begin{array}{ccc}
                   \tau_{1}^{(1)} & \tau_{2}^{(1)} & \dots  \\
                    \mu_1 & \mu_2 & \dots \\
                  \end{array}
                \right)$ ,
                 $\left(
                  \begin{array}{ccc}
                   \tau_{1}^{(2)} & \tau_{2}^{(2)} & \dots  \\
                    \nu_1 & \nu_2 & \dots \\
                  \end{array}
                \right)$ be the passports of $(\nabla_1, \mu)$ and $(\nabla_2, \nu)$.

The following conditions are equivalent:
\begin{enumerate}[(i)]
\item The   $\ast$-algebras $\mathcal L_{\log}(\nabla_1, \mu)$ and $\mathcal L_{\log}(\nabla_2, \nu)$ are $\ast$-isomorphic;
\item The upper rows of the passports of $(\nabla_1, \mu)$ and $(\nabla_2, \nu)$ coincide and the sequences 
$\{\frac{\mu_n}{\nu_n}\}$ and $\{\frac{\nu_n}{\mu_n}\}$ are bounded.
\end{enumerate}
\end{teo}
\begin{proof}
$(i)\ \Rightarrow\ (ii)$: Let $\Psi: \mathcal L_{\log}(\nabla_1, \mu) \to \mathcal L_{\log}(\nabla_2, \nu)$ be a
$\ast$-isomorphism. Then the restriction $\varphi$ of $\Psi$ on the Boolean algebra $\nabla_1$ is an isomorphism from  
$\nabla_1$ onto $\nabla_2$, hence, by Theorem \ref{t3}, the upper rows of passports of $(\nabla_1, \mu)$ and 
$(\nabla_2, \nu)$ coincide.

Let $\{e_n\}_{n=1}^\infty$ be a partition of a unity  $\mathbf 1_{\nabla_1} $ of the Boolean algebra $\nabla_1$ such that $e_n \nabla_1$ is a homogeneous Boolean algebra, $\tau_n^{(1)}=\tau (e_n\nabla_{1})< \tau_{n+1}^{(1)}, $ and 
$\mu_n = \mu(e_n)$ for each $n$. Set $q_n = \varphi(e_n)$. It is clear that $\{q_n\}_{n=1}^\infty$ 
is a partition of the unity  $\mathbf 1_{\nabla_2} $ of the Boolean algebra $\nabla_2$ such that $ q_n\nabla_2$ 
is a homogeneous Boolean algebra, $\tau_n^{(2)}=\tau ({q_n\nabla_2})< \tau_{n+1}^{(2)}, $ and 
$\nu_n = \nu(q_n)$ for each $n$.

By Proposition \ref{p2}, for a probability measure $\lambda(\varphi(e)) = \mu(e)$, $e \in \nabla_1$, on $\nabla_2$ we have 
$$
\mathcal L_{\log}(\nabla_2, \nu) = \Psi(\mathcal L_{\log}(\nabla_1, \mu)) = \mathcal L_{\log}(\nabla_2, \lambda).
$$
Using Corollary, \ref{c1} we see that
$$
\frac{d\nu}{d\lambda} \in \mathcal L_{\infty}(\nabla_2)  \  \ \text{and} \ \  \frac{d\lambda}{d\nu} \in \mathcal L_{\infty}(\nabla_2).
$$
Consequently,
$$
\nu_n = \nu(q_n)= \int \limits_{q_n}\frac{d\nu}{d\lambda}d\lambda \leq \|\frac{d\nu}{d\lambda}\|_\infty \lambda(q_n) = \|\frac{d\nu}{d\lambda}\|_\infty \mu(e_n)
$$
for all $n$. Therefore, the sequence $\{\frac{\nu_n}{\mu_n}\}$ is bounded. That the sequence 
$\{\frac{\mu_n}{\nu_n}\}$ is also bounded is shown similarly.

$(ii)\ \Rightarrow\ (i)$: Let the upper rows of the passports of $(\nabla_1, \mu)$ and $(\nabla_2, \nu)$ coincide and the sequences $\{\frac{\mu_n}{\nu_n}\}$ and $\{\frac{\nu_n}{\mu_n}\}$ are bounded. By Theorem \ref{t3}, there exists an isomorphism $\varphi: \nabla_1 \to \nabla_2$. Let $\{e_n\}_{n=1}^\infty$ and  $q_n = \varphi(e_n)$ be as in the proof of the implication $(i)\ \Rightarrow\ (ii)$.  Consider the probability measure $\gamma(q)=\sum \limits_{n=1}^{\infty}\mu_n \nu_n^{-1} \nu(q_n q)$, $q \in \nabla_2$, on  $\nabla_2$. Since the passports of the Boolean algebras  $(\nabla_1, \mu)$ and $(\nabla_2, \gamma)$ coincide, it follows by Theorem \ref{t3} that there exists an isomorphism  $\psi:  \nabla_1 \to  \nabla_2$ such that  $\mu(e) = \gamma (\psi (e))$ for all   $e \in \nabla_1$.

Let now $\Psi: \mathcal L_0(\nabla_1, \mu) \to \mathcal L_0(\nabla_2, \gamma)$ be a
$\ast$-isomorphism such that $\psi(e) = \Psi(e)$ for all $e \in \nabla_1$.
By Proposition \ref{p2}, 
$\Psi(\mathcal L_{\log}(\nabla_1, \mu)) = \mathcal L_{\log}(\nabla_2, \gamma).$ Since $\gamma(q)=\sum \limits_{n=1}^{\infty}\mu_n \nu_n^{-1} \nu(q_n q)$, $q \in \nabla_2$, it follows that
$$
\frac{d \gamma}{d \nu} = \sum_{n=1}^\infty \mu_n \nu_n^{-1} q_n \in \mathcal L_{\infty}(\nabla_2).
$$
On the other hand, if $q \in q_n \nabla_2$, then $\nu( q)=\nu(q_n q)=  \nu_n \mu_n^{-1} \gamma(q)$, that is, 
$$\nu(q)=\sum \limits_{n=1}^{\infty}\nu_n \mu_n^{-1}\gamma(q_n q), \ q \in \nabla_2.$$ Consequently,
$$
\frac{d\nu}{d \gamma} = \sum_{n=1}^\infty \nu_n \mu_n^{-1} q_n \in \mathcal L_{\infty}(\nabla_2).
$$
Therefore, by Corollary \ref{c1}, we have 
$$
\mathcal L_{\log}(\nabla_2, \nu)=\mathcal L_{\log}(\nabla_2, \gamma)=\Psi(\mathcal L_{\log}(\nabla_1, \mu)).
$$
\end{proof}

\end{document}